\documentclass[12pt]{article}
\usepackage[utf8]{inputenc}
\usepackage{lmodern}
\usepackage{graphicx, xcolor, mathrsfs}
\usepackage{amsfonts}
\usepackage{stmaryrd}
\usepackage{blkarray, bigstrut}
\usepackage{mathtools}
\usepackage[nice]{nicefrac}
\usepackage{amsmath,amsthm,amssymb,bbm}
\usepackage[mathlines]{lineno}
\usepackage{mathabx}
\usepackage{accents}

\usepackage{enumerate}
\usepackage{enumitem}
\usepackage[colorlinks=true,citecolor=black,linkcolor=black,urlcolor=blue]{hyperref}
\usepackage[top=.9in, bottom=.9in, left=.5in , right=.5in]{geometry}
\usepackage{caption}
\usepackage{xifthen} 
\usepackage{wrapfig}
\usepackage[numbers, square]{natbib}
\usepackage{changes}
\usepackage{comment}
\usepackage{float}
\usepackage{ifthen}
\mathtoolsset{showonlyrefs}
\usepackage{tikz}
\usepackage{caption}                       
\usetikzlibrary{arrows}
\usetikzlibrary{graphs,graphs.standard}
\usetikzlibrary{positioning,arrows.meta}
\usetikzlibrary{shapes.multipart}
\usetikzlibrary{arrows}
\usetikzlibrary{automata}
\definecolor{arrowblue}{RGB}{0,0,0}  

\usepackage[splitrule]{footmisc} 
\usepackage{caption}
\newtheorem{thm}{Theorem}[section]
\newtheorem*{thm*}{Theorem}
\newtheorem{lem}[thm]{Lemma}
\newtheorem{definition}[thm]{Definition}
\newtheorem{clm}{Claim}[thm]
\newtheorem{prop}[thm]{Proposition}
\newtheorem{cor}[thm]{Corollary}

\theoremstyle{definition}
\newtheorem{rmq}[thm]{Remark}

\newcommand{\E}[2][]{\ensuremath{\mathbb{E}_{#1}\left[#2 \right]}}
\newcommand{\Prob}[2][]{\ensuremath{\mathbb{P}_{#1} \left(#2 \right)}}
\newcommand{\N}{\mathbb{N}}
\def\Var{\mathop{\rm Var}\nolimits}
\newcommand{\var}[2][]{\ensuremath{\Var_{#1} \left(#2 \right)}}
\newcommand{\eps}{\varepsilon}
\newcommand{\dd}{\mathrm d}
\DeclareMathOperator{\exponentialrv}{Exp}

\newcommand{\ssup}[1] {{\scriptscriptstyle{({#1}})}}

\title{Permutations in competing growth processes and balls-in-bins}
\author{Johannes B\"aumler\footnote{Mathematical Institute, University of Koblenz, Germany. E-Mail: \href{mailto:jbaeumler@uni-koblenz.de}{jbaeumler@uni-koblenz.de} } \ and Tejas Iyer\footnote{Weierstrass Institute for Applied Analysis and Stochastics, Anton-Wilhelm-Amo-Str. 39, 10117 Berlin, Germany. \\ E-Mail: \href{mailto:tejas.iyer@wias-berlin.de}{tejas.iyer@wias-berlin.de} }}
\date{\today}

\begin{document}

\maketitle
\abstract{Consider a model of $N$ independent, increasing $\mathbb{N}_0$-valued processes, with random, independent waiting times between jumps. It is known that there is either an emergent `leader', in which a single process possesses the maximal value for all sufficiently large times, or every pair of processes alternates leadership infinitely often. We show that in the latter regime, almost surely, one sees every possible permutation of rankings of processes infinitely often. In the case that the waiting times are exponentially distributed, this proves a conjecture from Spencer (appearing in a paper from Oliveira) on the `balls-in-bins' process with feedback \cite[Conjecture 1]{oliveira-brownian-motion}.
}

\section{Introduction}

A natural model for the evolution of the wealth of entities over time is to consider competing birth processes. One can consider a fixed, finite number of `agents' with `values'  increasing in steps, from $j-1$ to $j$ after a random, non-negative amount of time $X_{j}$. In the case where the random variables $(X_{j})_{j \in \mathbb{N}}$ are independent, and identically distributed across agents, in~\cite{fixation-leadership-growth-processes}, the second author showed that with probability zero or one, a single individual becomes the \emph{leader}, possessing the maximum wealth for all sufficiently `large' times. In addition, the author showed that, in the regime of non-leadership, any two agents will fluctuate in order of value infinitely often.  

This result was a generalisation of previous results in the literature~\cite{khanin-neuron-polarity, oliveira-brownian-motion, oliveira-onset-of-dominance} which dealt with the case that the $(X_{j})_{j \in \mathbb{N}}$ are exponentially distributed random variables. Via a result commonly termed `Rubin's construction' in the literature~\cite{davis-rubin} (closely connected to the Athreya-Karlin embedding \cite{arthreya-karlin-embedding-68}), it is known that when the random variables are exponentially distributed, with $X_{j} \sim \exponentialrv((f(j-1))$, the collection of values of agents in the system, as the values change, behaves like the following discrete `balls-in-bins' process with feedback: at each new time step, a bin with $m$ balls is selected with probability proportional to $f(m)$ and a new ball is added to the bin.    

A natural conjecture is that, in the regime of non-leadership, given that we already know that any pair of agents fluctuates in ordering of value infinitely often, one in fact sees \emph{any} possible permutation of ordered values of agents infinitely often. In the context of balls-in-bins processes, this was conjectured by Spencer~\cite[Conjecture~1]{oliveira-brownian-motion}, stated in a paper from Oliveira. In this paper, we show that these conjectures hold.

We note that the results presented here have implications beyond urn models, for example, to the preferential attachment tree models analysed in~\cite{banerjee2021persistence, rudas} in regimes where there is no `leader' (coined \emph{persistent hubs} in this context) - see Remark~\ref{rem:pref-attach} further below. 

\subsection{Model description and main results}
We consider a finite collection of $\mathbb{N}_{0}$-valued growth processes with independent waiting times between jumps. Suppose we have $A \geq 2$ \emph{agents} labelled by the elements of $[A] := \left\{1, \ldots, A\right\}$. To each agent~$a \in [A]$, we associate an identically distributed sequence of mutually independent random variables $(X^{\ssup{a}}_{j})_{j \in \mathbb{N}}$, taking values in $[0, \infty)$, such that the sequences $(X^{\ssup{a}}_{j})_{j \in \mathbb{N}}$ are independent across different agents $a \in [A]$. At each time $t\geq 0$, each agent $a \in [A]$ has a \emph{value} $v_a(t) \in \mathbb{N}_0$ such that for each agent $a \in [A]$, its value $v_{a}\colon [0, \infty) \rightarrow \mathbb{N}_0$ increases over time. The random variable $X^{\ssup{a}}_{j}$ denotes the time taken for the value of agent $a$ to increase from $j-1$ to $j$. Additionally, to each agent $a \in [A]$, we associate an \emph{initial value} $v_{a}^{\text{in}} \in \mathbb{N}_0$. Thus, given the value $v_{a}^{\text{in}}$, for $k \in \mathbb{N}_{0}$ we have 
\[
v_a(t) = v_{a}^{\text{in}} + k \quad \text{if and only if} \quad \sum_{j=v_{a}^{\text{in}} + 1}^{v_{a}^{\text{in}} + k} X^{\ssup{a}}_{j} \leq t < \sum_{j=v_{a}^{\text{in}} + 1}^{v_{a}^{\text{in}} + k+1} X^{\ssup{a}}_{j}. 
\]  
Note that if $X_1^{\ssup{a}}>0$, then $v_a(0)=v_{a}^{\text{in}}$.
We are interested in the vector of values of agents, i.e., $\left(v_a(t)\right)_{a \in [A]}$, as time evolves. \\

Throughout this paper, for a random variable $Y$, we denote by $Y^{s}$ a random variable distributed like $Y-Y'$, where $Y'$ is an i.i.d. copy of $Y$. We denote by $\left(X^{s}_{j}\right)_{j\in \mathbb{N}}$ a sequence of independent random variables such that $X^{s}_{j}$ is distributed like $X^{\ssup{1}}_{j} - X^{\ssup{2}}_{j}$. 

Our main result is the following:

\begin{thm} \label{thm:main}
Suppose that the random series $\sum_{j=1}^{\infty}X^{s}_{j}$ diverges almost surely.
Then, almost surely, for any permutation $\pi\colon [A] \rightarrow [A]$, 
\begin{equation} \label{eq:perm-inequality}
\exists (t_{i})_{i \in \mathbb{N}} \in [0, \infty)^{\mathbb{N}}\colon \quad \lim_{i \to \infty} t_{i} = \infty \quad \text{and} \quad \forall i \in \mathbb{N} \quad 
v_{\pi(1)}(t_{i}) \geq v_{\pi(2)}(t_{i}) \geq \cdots \geq v_{\pi(A)}(t_{i}). 
\end{equation}
\end{thm}

It is a straightforward consequence of Kolmogorov's three series theorem (Theorem \ref{thm:3 series} below) that $\sum_{j=1}^\infty X_j^s$ diverges if and only if one of the sums $\sum_{j=1}^\infty \mathbb{P}(|X_j^s| \geq 1)$ or $\sum_{j=1}^\infty \mathbb{E}\left[ |X_j^s|^2 \mathbbm{1}_{|X_j^s|< 1} \right]$ diverges. The divergence of either sum implies that $\sum_{j=1}^\ell X_j^s$ is not concentrated on any finite interval as $\ell\to \infty$, see Theorem \ref{thm:dispersion} below. This anticoncentration is an essential input in our proof, and we use it to show that the long-time behaviour $(t \to \infty)$ of events like $\{v_{\pi(1)}(t) \geq v_{\pi(2)}(t) \geq \cdots \geq v_{\pi(A)}(t)\}$ are essentially invariant under initial time-shifts of some agents.

As outlined in the introduction, via Rubin's construction, Theorem \ref{thm:main} has implications for balls-in-bins processes with feedback. We recall the definition of such processes: we are given $A$ \emph{bins}, a \emph{feedback} function $f\colon \mathbb{N}_0 \rightarrow (0, \infty)$, and an initial collection of \emph{balls} in bins $(u_{a}(0))_{a \in [A]} \in \mathbb{N}^{A}$. Then, recursively, for $n \in \mathbb{N}$:
\begin{enumerate}
\item A bin $a \in [A]$ is sampled with probability 
			 \[\frac{f(u_{a}(n-1))}{\sum_{a \in [A]} f(u_a(n-1))}. \]
\item We set $u_{a}(n) = u_{a}(n-1) + 1$, whilst for $a' \neq a$, we set $u_{a'}(n) = u_{a'}(n-1)$.			 
\end{enumerate} 

Theorem \ref{thm:main} then directly implies the following corollary, which gives an affirmative answer to \cite[Conjecture~1]{oliveira-brownian-motion}.

\begin{cor}[{\cite[Conjecture~1]{oliveira-brownian-motion}}] \label{cor:conj}
    Consider a balls-in-bins process $(u_{a}(n))_{a \in [A], n \in \mathbb{N}_0}$ with feedback function $f\colon \mathbb{N}_0 \rightarrow (0, \infty)$ such that 
    \begin{equation} \label{eq:div-variance}
    \sum_{i=0}^{\infty} \frac{1}{f(i)^{2}} = \infty. 
    \end{equation} 
Then, almost surely, for any permutation $\pi\colon [A] \rightarrow [A]$, there exist infinitely many $n \in \mathbb{N}_0$ such that 
    \[
    u_{\pi(1)}(n) \geq u_{\pi(2)}(n) \geq \cdots \geq u_{\pi(A)}(n). 
    \]
\end{cor}

\begin{rmq} \label{rem:pref-attach}
The result in Corollary~\ref{cor:conj} has implications for \emph{preferential attachment trees} - where nodes arrive one at a time and connect to an existing node with probability proportional to a function $f$ of their out-degree (the model considered in, for example, \cite{banerjee2021persistence, rudas}). In particular, it shows that, if the function $f$ satisfies~\eqref{eq:div-variance}, then, given any finite collection of nodes, one sees any possible ordering of these nodes when ordered by degree infinitely often in the evolution of the tree. Theorem~\ref{thm:main} has a similar implication for the genealogical trees of CMJ branching processes with independent increments (for example, the model considered in~\cite{iyer2024persistenthubscmjbranching}), as long as the random series $\sum_{j=1}^{\infty}X^{s}_{j}$ diverges almost surely, and the model is `non-explosive'.
\end{rmq}

Finally, we note that it is natural to conjecture that a stronger version of Theorem \ref{thm:main} holds, which says that under the same hypothesis as in Theorem \ref{thm:main}, one actually has strict inequality in~\eqref{eq:perm-inequality} for any permutation $\pi:[A] \to [A]$. In the following proposition, we show that, perhaps surprisingly, this is not the case even for the balls-in-bins process, i.e., when the random variables $(X^{\ssup{a}}_{j})_{j \in \mathbb{N}}$ are exponentially distributed:

\begin{prop}[Counterexample to strict inequality in Theorem~\ref{thm:main}] \label{prop:counter}
For $A \geq 3$ there exists an example of a competing growth process with exponentially distributed random variables $(X^{\ssup{a}}_{j})_{j \in \mathbb{N}}$ such that $\sum_{j=1}^{\infty}X^{s}_{j}$ diverges almost surely, but it is almost surely not the case that 
\begin{equation}\label{eq:strong-conjecture}
\exists (t_{i})_{i \in \mathbb{N}} \in [0, \infty)^{\mathbb{N}}\colon \quad \lim_{i \to \infty} t_{i} = \infty \quad \text{and} \quad \forall i \in \mathbb{N} \quad 
v_{\pi(1)}(t_{i}) > v_{\pi(2)}(t_{i}) > \cdots > v_{\pi(A)}(t_{i}). 
\end{equation}
In particular, for $A \geq 3$ there exists a balls-in-bins process $(u_{a}(n))_{a \in [A], n \in \mathbb{N}_0}$ with feedback function $f\colon \mathbb{N}_0 \rightarrow (0, \infty)$ such that $\sum_{i=0}^{\infty} 1/f(i)^2 = \infty$, but it is not the case that there exists infinitely many $n \in \mathbb{N}_0$ such that 
\[u_{\pi(1)}(n) > u_{\pi(2)}(n) > \cdots > u_{\pi(A)}(n). \]
\end{prop}
\begin{rmq}
Note that under the assumption that $\sum_{j=1}^{\infty}X^{s}_{j}$ diverges almost surely~\eqref{eq:strong-conjecture} does in fact always hold for $A = 2$, and in particular for $A \geq 3$, almost surely, the comparative ranking of any pair of agents is strict infinitely often. This follows from~\cite[Proposition~1.17]{fixation-leadership-growth-processes} which shows that under the divergence condition $\limsup_{n \to \infty} \sum_{j=1}^{n}X^{s}_{j} = \infty$ and $\liminf_{n \to \infty} \sum_{j=1}^{n}X^{s}_{j} = -\infty$, so the comparative values of the agents changes signs infinitely often. It remains an interesting open problem to determine under which conditions \eqref{eq:strong-conjecture} does hold for $A \geq 3$.    
\end{rmq}

\section{Proofs of results}\label{sec:proof}
For the proof of Theorem~\ref{thm:main}, we first prove that the result can be reduced to Proposition~\ref{prop:main-new} below - the case that 
the initial values $\left(v_{a}^{\text{in}}\right)_{a\in [A]}$ are all identically zero, i.e., 
\begin{equation}\label{eq:starting points at 0}
    v_{a}^{\text{in}} = 0 \quad \text{ for all } a \in [A].
\end{equation}
\begin{prop} \label{prop:main-new}
Suppose that the random series $\sum_{j=1}^{\infty}X^{s}_{j}$ diverges almost surely and that $v_{a}^{\text{in}}=0$ for all $a \in [A]$. 
Then, almost surely, for any permutation $\pi\colon [A] \rightarrow [A]$, 
\begin{equation} \label{eq:perm-ordering}
\exists (t_{i})_{i \in \mathbb{N}} \in [0, \infty)^{\mathbb{N}}\colon \quad \lim_{i \to \infty} t_{i} = \infty \quad \text{and} \quad \forall i \in \mathbb{N} \quad 
v_{\pi(1)}(t_{i}) \geq v_{\pi(2)}(t_{i}) \geq \cdots \geq v_{\pi(A)}(t_{i}). 
\end{equation}
\end{prop}

We make this reduction in Section~\ref{sec:starting} below. The rest of Sections~\ref{subsec:propo proof} and~\ref{sec:propo proof} are dedicated to the proof of Proposition~\ref{prop:main-new}.

\subsection{Proof of Theorem \ref{thm:main} from Proposition~\ref{prop:main-new}}\label{sec:starting}
\begin{proof}[Proof of Theorem \ref{thm:main}]
    Let $(v_{a}^{\text{in}})_{a \in [A]} \in \N_0^{[A]}$ with $K \coloneqq \max_{a \in [A]} v_{a}^{\text{in}}$. Let $\big(X_i^{\ssup{a}}\big)_{i \in \N, a \in [A]}$ be random variables as described in Theorem \ref{thm:main} and let $\big(Z_i^{\ssup{a}}\big)_{i \in [K], a \in [A]}$ be i.i.d. random variables with
    \begin{equation} \label{eq:bernoulli-def}
        \mathbb{P} \left(Z_i^{\ssup{a}} = 0 \right) = \mathbb{P} \left( Z_i^{\ssup{a}} = 1 \right) = \frac{1}{2}
    \end{equation}
    that are furthermore independent of $\big(X_i^{\ssup{a}}\big)_{i \in \N, a \in [A]}$. Define the random variables $\big(Y_i^{\ssup{a}}\big)_{i \in \N, a \in [A]}$ by
    \begin{equation*}
        Y_i^{\ssup{a}} = \begin{cases}
            Z_i^{\ssup{a}} X_i^{\ssup{a}} & \text{ if } i \leq K \\
            X_i^{\ssup{a}} & \text{ if } i > K
        \end{cases} .
    \end{equation*}
    Define the processes $\left( v_a(t) \right)_{t\geq 0}$ and $\left( \tilde{v}_a(t) \right)_{t\geq 0}$ by
    \begin{align*}
        & v_a(t) = v_{a}^{\text{in}} + \sum_{j = v_{a}^{\text{in}} + 1}^{\infty} \mathbbm{1}_{\left\{ \sum_{i = v_{a}^{\text{in}} + 1}^{j} X_i^{\ssup{a}} \leq  t \right\}} \quad
        \text{ and  } \\
        & \tilde{v}_a(t) = \sum_{j = 1}^{\infty} \mathbbm{1}_{\left\{ \sum_{i = 1}^{j} Y_i^{\ssup{a}} \leq  t \right\}}  .
    \end{align*}
    Note that $\left( \tilde{v}_a(t) \right)_{a \in [A], t\geq 0}$ is a collection of competing birth processes, each starting from the common initial value $0$, and that $\sum_{i=1}^\infty \left( Y_i^{\ssup{1}} - Y_i^{\ssup{2}} \right)$ diverges almost surely, since
    \begin{equation*}
        \sum_{i=1}^\ell \left( Y_i^{\ssup{1}} - Y_i^{\ssup{2}} \right) = \sum_{i=1}^K \left( Y_i^{\ssup{1}} - Y_i^{\ssup{2}} \right) + \sum_{i=K+1}^\ell \left( X_i^{\ssup{1}} - X_i^{\ssup{2}} \right) \quad \text{ for $\ell > K$.}
    \end{equation*}
    Proposition~\ref{prop:main-new} thus implies that, almost surely, for any permutation $\pi\colon [A] \rightarrow [A]$, 
\begin{equation}\label{eq:tilde permutation sequence exists}
\exists (t_{i})_{i \in \mathbb{N}} \in [0, \infty)^{\mathbb{N}}\colon \quad \lim_{i \to \infty} t_{i} = \infty \quad \text{and} \quad \forall i \in \mathbb{N} \quad 
\tilde{v}_{\pi(1)}(t_{i}) \geq \tilde{v}_{\pi(2)}(t_{i}) \geq \cdots \geq \tilde{v}_{\pi(A)}(t_{i}). 
\end{equation}
    Define the event $\mathcal{I}$ by
    \begin{equation*}
        \mathcal{I} = \bigcap_{a \in [A]} \bigcap_{i=1}^{v_{a}^{\text{in}}} \left\{ Z_i^{\ssup{a}} = 0 \right\} \cap \bigcap_{a \in [A]} \bigcap_{i=v_{a}^{\text{in}} + 1}^{K} \left\{ Z_i^{\ssup{a}} = 1 \right\}.
    \end{equation*}
    By the definition of $\big(Z_i^{\ssup{a}}\big)_{i \in [K], a \in [A]}$ in Equation~\eqref{eq:bernoulli-def}, we get that $\mathbb{P}\left( \mathcal{I} \right) = 2^{-KA} > 0$. Further, conditioned on the event $\mathcal{I}$ we have that $\tilde{v}_a(t) = v_a(t)$ for all $a\in [A]$ and $t \geq 0$. Since the event defined in Equation~\eqref{eq:tilde permutation sequence exists} holds almost surely, it still holds almost surely after conditioning on $\mathcal{I}$. Thus, we get that, almost surely, for any permutation $\pi\colon [A] \rightarrow [A]$, 
\begin{equation}\label{eq:tilde permutation sequence exists }
\exists (t_{i})_{i \in \mathbb{N}} \in [0, \infty)^{\mathbb{N}}\colon \quad \lim_{i \to \infty} t_{i} = \infty \quad \text{and} \quad \forall i \in \mathbb{N} \quad 
v_{\pi(1)}(t_{i}) \geq v_{\pi(2)}(t_{i}) \geq \cdots \geq v_{\pi(A)}(t_{i}),
\end{equation}
which proves Theorem \ref{thm:main}.
\end{proof}

\subsection{Proof of Proposition~\ref{prop:main-new}}\label{subsec:propo proof}
We write $S_A$ for the symmetric group on $[A]$. For a permutation $\pi:[A] \to [A]$ and a collection of integers $(M_a)_{a \in [A]}$, we define
\begin{equation*}
	\Xi_\pi \left( (M_a)_{a \in [A]} \right) = \frac{1}{\sum_{\rho \in S_A} \mathbbm{1} \left\{M_{\rho(1)} \geq M_{\rho(2)} \geq \ldots \geq M_{\rho(A)}\right\} } \mathbbm{1} \left\{M_{\pi(1)} \geq M_{\pi(2)} \geq \ldots \geq M_{\pi(A)}\right\}.
\end{equation*}
By definition, the normalising factor means that
\begin{align*}
	\sum_{\pi \in S_A} \Xi_\pi \left( (M_a)_{a \in [A]} \right) = 1.
\end{align*}
We will evaluate the function $\Xi_\pi$ at the values $\left(v_a(t)\right)_{a \in [A]}$. The following proposition shows that the long-term behaviour of the expectation of $\Xi_\pi\big( \left(v_a(t)\right)_{a \in [A]} \big)$ is not affected by initial time-shifts:

\begin{prop}\label{prop:A factorial conv}
    Suppose that the random series $\sum_{j=1}^{\infty}X^{s}_{j}$ diverges almost surely and that $v_{a}^{\text{in}}=0$ for all $a \in [A]$. 
	Let $\pi \in S_A$ and let $M>0$. Then
	\begin{equation}\label{limit t arbitrary starts}
		\lim_{t\to \infty} \ \sup_{s_1,\ldots,s_A \in [0,M]} \ \left| \mathbb{E} \left[ \Xi_\pi \left( \left(v_a(t+s_a)\right)_{a\in [A]} \right) \right] - \frac{1}{A!} \right| = 0 .
	\end{equation}
\end{prop}
The above equality is clear when $s_a=0$ for all $a \in [A]$. In this case, symmetry implies that $\mathbb{E} \left[ \Xi_\pi \left( \left(v_a(t+s_a)\right)_{a\in [A]} \right) \right] = \frac{1}{A!}$ for all $t\geq 0$.

As $\Xi_\pi \left( \left(v_a(t+s_a)\right)_{a\in [A]} \right) \leq  \mathbbm{1} \left\{ v_{\pi(1)}(t+s_{\pi(1)}) \geq \ldots \geq v_{\pi(A)}(t+s_{\pi(A)}) \right\}$, the above proposition directly implies the following.

\begin{cor}\label{cor convergence}
    Let $\pi \in S_A$, let $M>0$, and let $s_1,\ldots,s_A \in [0,M]$. Then
	\begin{equation}
		\liminf_{t\to \infty} \ \inf_{s_1,\ldots,s_A \in [0,M]}  \ \mathbb{P} \left(  v_{\pi(1)}(t+s_{\pi(1)}) \geq \ldots \geq v_{\pi(A)}(t+s_{\pi(A)})  \right) \geq  \frac{1}{A!}  .
	\end{equation}
\end{cor}

We defer the proof of Proposition \ref{prop:A factorial conv} to Section~\ref{sec:propo proof}, first using it to prove Proposition~\ref{prop:main-new}. In what follows, we define $\tau_{a}(n)$ to be the time taken for agent $a \in [A]$ to reach value $n$, i.e. 
\begin{equation} \label{eq:stopping-time}
    \tau_{a}(n) := \inf\left\{t \geq 0\colon v_{a}(t) \geq n \right\} = \begin{cases}
\sum_{j= 1}^{n} X^{\ssup{a}}_{j} & \text{if $n \in \mathbb{N}$} \\
0 & \text{otherwise. }
    \end{cases}    
\end{equation}
We also define the increasing sequence of $\sigma$-algebras $(\mathcal{F}_t)_{t\geq 0}$ by 
\[
    \mathcal{F}_{t} := \sigma(v_a(s) \colon a \in [A], s \leq t).
\]

\begin{proof}[Proof of Proposition~\ref{prop:main-new}]
It suffices to prove that almost surely, Equation~\eqref{eq:perm-ordering} is satisfied for the trivial permutation $\pi(i) \equiv i$. This implies by symmetry that Equation~\eqref{eq:perm-ordering} is almost surely satisfied for any given permutation, and thus, taking the intersection over the finitely many permutations possible, the result follows. Define the event $\mathcal{E}_t$ by
    \[
    \mathcal{E}_t := \left\{ v_1(t) \geq v_2(t) \geq \ldots \geq v_A(t) \right\} .
    \]
We start with the following claim:
    \begin{clm} \label{clm-1}
    For all $\theta \in [0, \infty)$, there exists an $\mathcal{F}_{\theta}$-measurable and almost surely finite random variable $Z$ such that
    \begin{equation*}
        \mathbb{P} \left( \mathcal{E}_t | \mathcal{F}_{\theta} \right) \geq \frac{1}{4A!}\mathbbm{1}_{\{t \geq Z\}}, \quad \text{almost surely}.
    \end{equation*}
    In particular, for all $\theta \in [0,\infty)$,
    \begin{equation} \label{eq:lim-next-time}
        \lim_{t \to \infty} \mathbb{P} \left( \mathbb{P} \left( \mathcal{E}_{t} | \mathcal{F}_{\theta} \right) < \frac{1}{4A!} \right) = 0 .
    \end{equation}
    \end{clm} 
    \begin{proof}
        Let $N = \max_{a \in [A]} v_a(\theta)$. Note that, since each of the values of $X^{\ssup{a}}_{i}$ are almost surely finite, $\max_{a \in [A]} \tau_a(N+1) < \infty$ almost surely. Therefore, 
        we can find an $\mathcal{F}_{\theta}$ measurable random variable $M$ such that $M < \infty$ almost surely and
        \begin{equation*}
            \mathbb{P} \left( \max_{a \in [A]} \tau_a(N+1) \leq M \big| \mathcal{F}_\theta \right) > \frac{1}{2}.
        \end{equation*}
        Define $w_a(s)$ by 
        \begin{equation*}
            w_a(s) = \sum_{j=N+2}^{\infty} \mathbbm{1}_{\left\{ s \geq \sum_{i=N+2}^{j} X_i^{\ssup{a}}  \right\} }.
        \end{equation*}
        The collection $\left( w_a(s) \right)_{a\in [A], s \geq 0}$ is also a collection of competing birth processes, 
        where $A$ agents have a value $w_a\colon \left[0,\infty\right) \to \mathbb{N}_0$ and the time taken for the value of agent $a$ to go from $k$ to $k+1$ is given by $X_{N+2+k}^{\ssup{a}}$. Since $\sum_{k=N+2}^{\infty}X_k^{s}$ diverges almost surely, we can apply Corollary~\ref{cor convergence} to $\left( w_a(s) \right)_{a\in [A], s \geq 0}$. In particular, this implies that we can choose an  $\mathcal{F}_{\theta}$ measurable random variable $Z>M$ such that
        \begin{equation*}
            \inf_{s_1,\ldots,s_A \in [0,M]} \  \mathbb{P} \left(  w_1(t+s_1) \geq \ldots \geq w_A(t+s_A) \big| \mathcal{F}_\theta  \right) \geq  \frac{1}{2A!} \mathbbm{1}_{\{t \geq Z\}}.
        \end{equation*}
        The dependence on $\mathcal{F}_\theta$ in the above conditional probability comes from the dependence of $(w_a(t+s_a))_{a \in [A]}$ on $N$. Note that by definition $v_a(\cdot)$ and $w_a(\cdot)$ satisfy 
        \begin{equation*}
            v_a \left( s + \tau_{a} (N+1) \right) = N+1 + w_a(s).
        \end{equation*}
        Since the $\sigma$-algebra $\mathcal{F}_{\theta}$ does not contain any information about $\left(X_j^{\ssup{a}} : j \geq N+2 , a \in [A] \right)$, we see that on $t \geq Z$ one has
        \begin{align*}
            & \mathbb{P} \left( \mathcal{E}_{t+M} \big| \mathcal{F}_{\theta} \right) \geq \mathbb{P} \left( \max_{a \in [A]} \tau_a(N+1) < M , v_1(t+M) \geq \ldots \geq v_A(t+M) \big| \mathcal{F}_{\theta} \right) 
            \\
            &
            = 
            \mathbb{P} \left( \max_{a \in [A]} \tau_a(N+1) < M , w_1(t+M - \tau_1(N+1)) \geq \ldots \geq w_A(t+M - \tau_A(N+1) ) \big| \mathcal{F}_{\theta} \right) 
            \\
            &
            \geq
            \inf_{s_1,\ldots,s_A \in [0,M]}
            \mathbb{P} \left( \max_{a \in [A]} \tau_a(N+1) < M , w_1(t + s_1) \geq \ldots \geq w_A(t + s_A ) \big| \mathcal{F}_{\theta} \right) 
            \\
            &
            =
            \left(
            \mathbb{P} \left( \max_{a \in [A]} \tau_a(N+1) < M  \big| \mathcal{F}_{\theta} \right) \right) \left( \inf_{s_1,\ldots,s_A \in [0,M]}
            \mathbb{P} \left( w_1(t + s_1) \geq \ldots \geq w_A(t + s_A ) \big| \mathcal{F}_{\theta} \right) \right) 
            \\
            &
            \geq \frac{1}{2} \cdot \frac{1}{2 A!} = \frac{1}{4A!}, \quad \text{almost surely,}
        \end{align*}
        where, in the second to last line, we use the conditional independence of the associated random variables given $\mathcal{F}_{\theta}$. 
    \end{proof}
    Choose $t_1 = 1$. Given $t_k$, by applying Equation~\eqref{eq:lim-next-time} from Claim~\ref{clm-1}, we choose $t_{k+1} \geq t_k+1$ such that
    \begin{equation}\label{eq:ksquared}
        \mathbb{P}  \left( \mathbb{P} \left( \mathcal{E}_{t_{k+1}} \big| \mathcal{F}_{t_k} \right) < \frac{1}{4A!} \right) < \frac{1}{k^2}.
    \end{equation}
    Since $t_{k+1} \geq t_k + 1$, this immediately implies that $\lim_{k\to \infty} t_k = \infty$.
    Further, $\sum_{k=1}^{\infty} \mathbb{P}\left( \mathbb{P} \left( \mathcal{E}_{t_{k+1}}| \mathcal{F}_{t_k} \right) < \frac{1}{4A!} \right) < \infty$ so that the first Borel-Cantelli lemma implies
    \begin{equation*}
        \mathbb{P}  \left( \mathbb{P} \left( \mathcal{E}_{t_{k+1}} | \mathcal{F}_{t_k} \right) < \frac{1}{4A!} \text{ for infinitely many } k \in \mathbb{N} \right) = 0 .
    \end{equation*}
    Consequently, 
    \begin{equation}\label{eq:infinite sum}
        \mathbb{P}  \left( \sum_{k=1}^{\infty} \Prob{\mathcal{E}_{t_{k+1}} |\mathcal{F}_{t_k}} = \infty \right) = 1. 
    \end{equation}
    L\'evy's extension of the Borel-Cantelli lemma, see \cite[Theorem 12.15, page 124]{williams1991probability}, now implies that
    \begin{equation*}
        \mathbb{P}  \left(  \mathcal{E}_{t_{k}} \text{ for infinitely many } k \in \mathbb{N} \right) =  1,
    \end{equation*}
    completing the proof of Proposition~\ref{prop:main-new}.
\end{proof}

\subsection{Proof of Proposition \ref{prop:A factorial conv}}\label{sec:propo proof}

A key tool for the proof of Proposition~\ref{prop:A factorial conv} is the inequality stated in Theorem~\ref{thm:dispersion} below.  This inequality gives quantitative bounds on the dispersion of random walks with independent increments. For a real-valued random variable $Y$, we define
\[
D(Y;\lambda) := \frac{1}{\lambda^{2}} \E{Y^{2} \mathbbm{1}_{|Y|\leq \lambda}} + \Prob{|Y| \geq \lambda}. 
\]
and 
\begin{equation} \label{eq:q-def}
Q(Y;\lambda) := \sup_{x \in \mathbb{R}} \Prob{x \leq Y \leq x+ \lambda}. 
\end{equation}

Theorem~\ref{thm:dispersion} is a known result from~\cite{petrov1995limit}, slightly reformulated and simplified for our purpose:

\begin{thm}[{\cite[Theorem~2.14, page 64]{petrov1995limit}}] \label{thm:dispersion}
    Let $Y_{1}, \ldots, Y_{n}$ be independent real-valued random variables, and $S_{n} := \sum_{i=1}^{n} Y_{i}$. Let $\lambda > 0$ be given. Then, there exists an absolute constant $B > 0$ such that
    \begin{equation} \label{eq:conc-bound}
    Q(S_{n}, \lambda) \leq B \left(\sum_{k=1}^{n} D(Y^{s}_{k}; \lambda)\right)^{-1/2}. 
    \end{equation}
\end{thm}
We remark that Theorem~\ref{thm:dispersion} is proved by using analytic methods to bound the absolute values of characteristic functions of the associated random variables. We also recall the well-known criteria providing necessary and sufficient conditions for a series of independent random variables to converge:
\begin{thm}[Kolmogorov three series theorem, e.g.{~\cite[Theorem~2.5.8., page~85]{durrett}}] \label{thm:3 series}
 For a sequence of mutually independent random variables $(S_{j})_{j \in \mathbb{N}}$, let $C > 0$ be given. Then the series $\sum_{j=1}^{\infty} S_{j}$ converges almost surely if and only if
 \begin{equation} \label{eq:kolmogorov-three-series}
     \sum_{j=1}^{\infty} \Prob{\left|S_{j}\right| > C} < \infty, \quad \sum_{j=1}^{\infty} \E{S_{j} \mathbbm{1}_{\left|S_{j}\right| \leq C}} < \infty,  \quad \text{ and } \quad \sum_{j=1}^{\infty} \var{S_{j} \mathbbm{1}_{\left|S_{j}\right| \leq C}} < \infty .
 \end{equation}
\end{thm}
Note that for the random series $\sum_{j=1}^{\infty} X^{s}_{j}$ and any $\lambda > 0$ one has, by symmetry of the associated random variables, that
\begin{equation*}
    \mathbb{E} \left[ X_j^s \mathbbm{1}_{|X_j^s|\leq \lambda} \right] = 0 \quad \text{ and } \quad \var{X_j^s \mathbbm{1}_{\left|X_j^s\right| \leq \lambda}} = \mathbb{E} \left[ (X_j^s)^2 \mathbbm{1}_{|X_j^s| \leq \lambda} \right].
\end{equation*}
Thus the Kolmogorov three series theorem implies that $\sum_{j=1}^{\infty} X^{s}_{j}$ diverges almost surely if and only if, for any $\lambda > 0$,
\begin{equation} \label{eq:div-imp-d-sum-div}
\sum_{j=1}^{\infty} D(X^{s}_{j}; \lambda) = \infty. 
\end{equation}

\begin{definition}
    For two real numbers $a,b \in \mathbb{R}$, we say that a function $h:[a,b] \to \mathbb{R}$ is {\sl unimodal} if there exists $t \in [a,b]$ such that $h$ is non-decreasing on $[a,t]$ and non-increasing on $[t,b]$.
\end{definition}

\begin{lem}
	Let $h:\mathbb{R} \to \left[0,1\right]$ be an increasing function, let $s>0$, and let $Y$ be a random variable. Then
	\begin{equation}\label{1-dim}
		\left|\mathbb{E}\left[h(Y+s)\right] - \mathbb{E}\left[h(Y)\right]\right| \leq Q(Y;s) .
	\end{equation}
    In particular, if $\tilde{h}\colon[a,b] \to \left[0,1\right]$ is unimodal and $\mathbb{P}\left(Y \in \left[ a,b - s \right] \right) = 1$, then
	\begin{equation}\label{unimodal}
		\left|\mathbb{E}\left[\tilde{h}(Y+s)\right] - \mathbb{E}\left[\tilde{h}(Y)\right]\right| \leq 2 Q(Y;s) .
	\end{equation}
\end{lem}

\begin{proof}
    We start with the proof of \eqref{1-dim} when $h$ is strictly increasing. Let $\nu$ be the probability measure on $\mathbb{R}$ for which $\nu(A)=\mathbb{P}(Y\in A)$ for all open sets $A \subset \mathbb{R}$. Fubini's Theorem implies that
    \begin{align}
        &\notag  \mathbb{E}\left[h(Y+s)\right] - \mathbb{E}\left[h(Y)\right] = \int h(y+s)-h(y) \dd \nu (y) 
        \\
        & \label{insert h star}
        =  \int \int_0^1 \mathbbm{1}_{ \{h(y) \leq x \leq h(y+s) \}} \dd x  \dd \nu (y) 
        =
        \int_0^1 \int \mathbbm{1}_{ \{h(y) \leq x \leq h(y+s) \}} \dd \nu (y) \dd x .
    \end{align}
    Since $h$ is increasing, the set $I_x \coloneqq \left\{y: h(y) \leq x \leq h(y+s)\right\}$ is an interval. Since $h$ is also strictly increasing, we can define
    \begin{align*}
    	h^\star (x) = \sup \left\{ y \in \mathbb{R} : h(y) \leq x \right\} = \inf \left\{ y \in \mathbb{R} : h(y) \geq x \right\}.
    \end{align*}
    If $z \in I_x$, then $h(z)\leq x$ and thus $z \leq h^{\star}(x)$. Also, if $z \in I_x$, then $h(z+s)\geq x$ and thus $z+s \geq h^{\star}(x)$, or equivalently $z \geq h^\star(x)-s$. Thus we see that $I_x \subseteq \left[h^\star(x)-s,h^\star(x)\right]$. Inserting this into \eqref{insert h star}, we see that
    \begin{align*}
    	& \mathbb{E}\left[h(Y+s)\right] - \mathbb{E}\left[h(Y)\right] = \int_0^1 \int \mathbbm{1}_{ \{h(y) \leq x \leq h(y+s) \}} \dd \nu (y) \dd x
    	\leq
    	\int_0^1 \int \mathbbm{1}_{ \{ h^\star(x)-s \leq y \leq h^\star(x) \}} \dd \nu (y) \dd x
    	\\ 
    	&
    	=
    	\int_0^1 \mathbb{P} \left( h^\star(x)-s \leq Y \leq h^\star(x) \right) \dd x
    	\stackrel{\eqref{eq:q-def}}{\leq}
    	\int_0^1 Q(Y;s) \dd x \leq Q(Y;s) .
    \end{align*}
    This finishes the proof for the case where $h$ is strictly increasing.
    When $s>0$ and $h$ is increasing, but not necessarily strictly increasing, define the functions $h_\eps:\mathbb{R}\to \left[0,1\right]$ by
    \begin{equation*}
    	h_\eps(x) = (1-\eps) h(x) + \eps \frac{1}{1+e^{-x}} .
    \end{equation*}
    For each $\eps > 0$, the function $h_\eps: \mathbb{R} \to [0,1]$ is strictly increasing, and $\lim_{\eps \searrow 0} h_{\eps} = h$ pointwise. Thus, we can use the previous argument for strictly increasing functions to get that $\mathbb{E}\left[h_\eps(Y+s)\right] - \mathbb{E}\left[h_\eps(Y)\right] \leq Q(Y;s)$. By the dominated convergence theorem, since $|h_\eps|,|h|\leq 1$, we see that
    \begin{equation*}
    	\mathbb{E}\left[h(Y+s)\right] - \mathbb{E}\left[h(Y)\right] = \lim_{\eps \searrow 0} \left(\mathbb{E}\left[h_\eps(Y+s)\right] - \mathbb{E}\left[h_\eps(Y)\right] \right) \leq Q(Y;s) .
    \end{equation*}
    The proof of \eqref{unimodal} follows once we observe that every unimodal function $\tilde{h} : [a,b] \to \left[0,1\right]$ can be written as the difference of two increasing functions $h_1, h_2 : [a,b] \to \left[0,2\right]$. Thus, using \eqref{1-dim}, we get that
    \begin{align*}
        &\left| \mathbb{E} \left[ \tilde{h}(Y+s) - \tilde{h}(Y) \right] \right| = \left| \mathbb{E} \left[ h_1(Y+s) - h_1(Y) \right] - \mathbb{E} \left[ h_2(Y+s) - h_2(Y) \right] \right|
        \\
        &
        \leq
        \max \left\{  \mathbb{E} \left[ h_1(Y+s) - h_1(Y) \right] , \mathbb{E} \left[ h_2(Y+s) - h_2(Y) \right] \right\}
        \\
        & =
        2 \max \left\{  \mathbb{E} \left[ \frac{h_1(Y+s)}{2} - \frac{h_1(Y)}{2} \right] , \mathbb{E} \left[ \frac{h_2(Y+s)}{2} - \frac{h_2(Y)}{2} \right] \right\}\leq 2 Q(Y;s),
    \end{align*}
    since $h_1/2$ and $h_2/2$ are increasing functions from $[a,b]$ to $[0,1]$.
    \end{proof}

    \begin{proof}[Proof of Proposition \ref{prop:A factorial conv}]
        Without loss of generality, we can assume that $\pi: [A]\to [A]$ is the identity -- all other cases follow by symmetry. Recall $\tau_a(n)= \sum_{j=1}^{n} X_j^{\ssup{a}}$ for $n \in \mathbb{N}$ and $a \in [A]$. Define the event $\mathcal{A}_t$ by
        \begin{equation*}
            \mathcal{A}_t = \bigcap_{a\in [A]} \left\{ \tau_{a}(n) \leq t  \right\}.
        \end{equation*}
        By symmetry, and the fact that $v_{a}^{\text{in}} = 0$ for all $a\in [A]$, it follows that $\mathbb{E} \left[ \Xi_\pi \left( v_a(t)_{a\in [A]} \right) \big| \mathcal{A}_t \right] = \frac{1}{A!}$. 
        The following claim quantifies the influence of initial time-shifts on the vector of values $\left(v_a(t)\right)_{a \in [A]}$.
        \begin{clm}\label{claim:variation}
        Let $n$ and $t$ be such that $\mathbb{P}\left( \mathcal{A}_t \right) > \frac{1}{2}$ and let $s_1,\ldots,s_A \geq 0$. Then
        \begin{equation} \label{eq:a-fact-conv-claim}
            \left| \mathbb{E} \left[ \Xi_\pi \left( v_a(t+s_a)_{a\in [A]} \right) \big| \mathcal{A}_t \right] - \mathbb{E} \left[ \Xi_\pi \left( v_a(t)_{a\in [A]} \right) \big| \mathcal{A}_t \right] \right|
            \leq \sum_{a \in [A]} 4 Q\left( \tau_{a}(n); s_a \right) .
        \end{equation}
    \end{clm}
        To use Claim~\ref{claim:variation} to complete the proof of Proposition~\ref{prop:A factorial conv}, let $\eps \in (0,1/2)$ be given. Since $\sum_{j=1}^{\infty} X^{s}_{j}$ diverges almost surely, Equation~\eqref{eq:div-imp-d-sum-div} and Theorem \ref{thm:dispersion} imply that we can fix $n \in \mathbb{N}$ sufficiently large that $\sum_{a \in [A]} 4 Q\left( \tau_{a}(n); M \right) < \eps$. Given such a choice of $n$, choose $t= t(n,\eps)$ sufficiently large that $\mathbb{P}\left( \mathcal{A}_t \right) > 1 - \eps$.
        Then
        \begin{align*}
            & \left| \mathbb{E} \left[ \Xi_\pi \left( v_a(t+s_a)_{a\in [A]} \right) \right] - \frac{1}{A!} \right|
            \\
            &
            \leq
            \mathbb{P} \left( \mathcal{A}_t \right) \left| \mathbb{E} \left[ \Xi_\pi \left( v_a(t+s_a)_{a\in [A]} \right) \big| \mathcal{A}_t \right] - \frac{1}{A!} \right| 
            + 
            \mathbb{P} \left( \mathcal{A}_t^c \right) \left| \mathbb{E} \left[ \Xi_\pi \left( v_a(t+s_a)_{a\in [A]} \right) \big| \mathcal{A}_t^c \right] - \frac{1}{A!} \right|
            \\
            &
            \leq
            \left| \mathbb{E} \left[ \Xi_\pi \left( v_a(t+s_a)_{a\in [A]} \right) \big| \mathcal{A}_t \right] - \frac{1}{A!} \right| 
            + 
            \mathbb{P} \left( \mathcal{A}_t^c \right) 
            \stackrel{\eqref{eq:a-fact-conv-claim}}{\leq}
            \sum_{a \in [A]} 4 Q \left( \tau_{a}(n) ; s_a \right)
            + 
            \mathbb{P} \left( \mathcal{A}_t^c \right) 
            \\
            &
            \leq
            \sum_{a \in [A]} 4 Q \left( \tau_{a}(n) ; M \right)
            + 
            \mathbb{P} \left( \mathcal{A}_t^c \right) \leq 2 \eps ,
        \end{align*}
        where we used that $s_1,\ldots,s_A \in [0,M]$ in the second to last inequality.
        As $\eps \in (0,1/2)$ was arbitrary, this finishes the proof of Proposition \ref{prop:A factorial conv}.
    \end{proof}
    It remains to prove Claim~\ref{claim:variation}.
    \begin{proof}[Proof of Claim~\ref{claim:variation}]
        Let $(Y_a)_{a \in [A]}$ be random variables that have the distribution of $(\tau_{a}(n))_{a \in [A]}$ conditioned on $\mathcal{A}_t$. 
        These random variables are still independent and identically distributed. Further, they satisfy
        \begin{equation}\label{eq:factor 2}
            Q(Y_a; s ) = \sup_{x} \mathbb{P} \left( x \leq \tau_{a}(n) \leq x+s \big| \mathcal{A}_t \right) \leq \sup_{x} \frac{\mathbb{P} \left( x \leq \tau_{a}(n) \leq x+s \right)}{\mathbb{P}\left( \mathcal{A}_t \right)} \leq 2 Q\left( \tau_{a}(n) ; s \right), 
        \end{equation}
        where we used $\mathbb{P}\left( \mathcal{A}_t \right) > 1/2$ for the last inequality.
        Define a new process $\left( \Tilde{v}_a (s) \right)_{a \in [A]}$ by
        \begin{equation*}
            \Tilde{v}_a(s) = \sum_{i=n+1}^{\infty} \mathbbm{1}_{\left\{ Y_a + \sum_{j=n+1}^{i} X_j^{\ssup{a}} \leq s \right\} } .
        \end{equation*}
        It directly follows that
        \begin{align*}
            &\mathbb{E} \left[ \Xi_\pi \left( v_a(t+s_a)_{a\in [A]} \right) \big| \mathcal{A}_t \right] = \mathbb{E} \left[ \Xi_\pi \left( \left(\tilde{v}_a(t+s_a)\right)_{a\in [A]} \right)  \right] \quad \text{ and}
            \\
            &
            \mathbb{E} \left[ \Xi_\pi \left( v_a(t)_{a\in [A]} \right) \big| \mathcal{A}_t \right] = \mathbb{E} \left[ \Xi_\pi \big( \Tilde{v}_a(t)_{a\in [A]} \big)  \right] .
        \end{align*}
        Fix $b \in [A]$ and define the $\sigma$-algebra $\mathcal{G}_{b}$ by
        \begin{equation*}
            \mathcal{G}_{b} = \sigma \big( \left( Y_a \right)_{ a \in [A]\setminus \{b\}} , \big( X_i^{\ssup{a}} \big)_{a \in [A], i > n} \big).
        \end{equation*}
        For fixed $\left( Y_a \right)_{ a \in [A]\setminus \{b\}} , \big( X_i^{\ssup{a}} \big)_{a \in [A], i > n}$ and $(s_a)_{a \in [A]}$, the function 
        \begin{equation*}
            x \mapsto \mathbb{E} \left[ \Xi_\pi \left( \left(\tilde{v}_a(t+s_a)\right)_{a\in [A]} \right) \big| \mathcal{G}_b, Y_b = x \right] 
        \end{equation*}
        only depends on $t+s_b-x$ and is unimodal on the domain $[0,t+s_b]$. To see the unimodality, note that the values $\left( \tilde{v}_a(t+s_a) \right)_{a\neq b}$ are measurable with respect to $\mathcal{G}_b$ and that $\tilde{v}_b(t+s_b)$ is non-increasing in $Y_b$ and measurable given $Y_{b}$ and $\mathcal{G}_{b}$. Thus, it suffices to show that
        \begin{equation*}
            k \mapsto \mathbb{E} \left[ \Xi_\pi \left( \left(\tilde{v}_a(t+s_a)\right)_{a\in [A]} \right) \big| \left( \tilde{v}_a(t+s_a) \right)_{a\neq b},  \tilde{v}_b(t+s_b) = k \right] 
        \end{equation*}
        is unimodal in $k$.
        Conditioned on $\left( \tilde{v}_a(t+s_a) \right)_{a\neq b}$ and on $\tilde{v}_b(t+s_b)$, we have that
        \begin{equation*}
            \mathbb{E} \left[ \Xi_\pi \left( \left(\tilde{v}_a(t+s_a)\right)_{a\in [A]} \right) \big| \left( \tilde{v}_a(t+s_a) \right)_{a\neq b},  \tilde{v}_b(t+s_b) = k \right] =  \Xi_\pi \left( \left(\tilde{v}_a(t+s_a)\mathbbm{1}_{a \neq b} + k\mathbbm{1}_{a = b} \right)_{a\in [A]} \right).
        \end{equation*}
        Thus, it suffices to show that for a collection of integers $(x_a)_{a \in [A]}$ the function
        \begin{equation}\label{eq:unimodal function}
            k \mapsto \Xi_\pi \left( \left(x_a\mathbbm{1}_{a \neq b} + k\mathbbm{1}_{a = b} \right)_{a\in [A]} \right)
        \end{equation}
        is unimodal. Recall that we assume $\pi(i)\equiv i$. 
        The function defined in \eqref{eq:unimodal function} is always zero when the values $(x_{a})_{a\neq b}$ disallow the permutation to take place. When the values $(x_{a})_{a\neq b}$ allow the permutation to take place: 
        \begin{itemize}
            \item When $b = 1$, the function is increasing, since, once the ranking becomes possible, i.e., when $k=\max_{a \neq 1} x_a$, there are fewer permutations allowing the ranking for $k > \max_{a \neq 1} x_a$.
            \item  Similarly, when $b=A$, the function defined in \eqref{eq:unimodal function} is decreasing in $k$.
            \item When $b \in \{2,\ldots,A-1\}$, and $x_{b-1} + 1 < x_{b+1}$ then the function defined in \eqref{eq:unimodal function} is zero for $k < x_{b-1}$, non-decreasing on $x_{b-1} \leq k < x_{b+1}$, decreases at $x_{b+1}$ and again at $x_{b+1} +1$, where it drops to zero.
            \item When $b \in \{2,\ldots,A-1\}$, and $|x_{b+1}-x_{b-1}|\leq 1$, then the function defined in \eqref{eq:unimodal function} is zero for $k \notin \{x_{b-1},x_{b+1}\}$ and positive for $k \in \{x_{b-1},x_{b+1}\}$, which directly implies unimodality.
        \end{itemize}
        So in particular, for fixed $\left( Y_a \right)_{ a \in [A]\setminus \{b\}} , \big( X_i^{\ssup{a}} \big)_{a \in [A], i > n}$ and $(s_a)_{a \in [A]}$, we can write 
        \begin{equation*}
            \mathbb{E} \left[ \Xi_\pi \left( \left(\tilde{v}_a(t+s_a)\right)_{a\in [A]} \right) \big| \mathcal{G}_b, Y_b = x \right] = g(t+s_b - x ) , \quad x \in [0,t+s_b] ,
        \end{equation*}
        where the function $g$ is unimodal on $[0,t+s_b]$. We define $\left(\Tilde{s}_a\right)_{a\in [A]}$ by $\Tilde{s}_a = s_a \mathbbm{1}_{\{a \neq b \}}$.
        Since the random variable $Y_b$ is supported on $[0,t]$, for fixed $\left( Y_a \right)_{ a \in [A]\setminus \{b\}} , \big( X_i^{\ssup{a}} \big)_{a \in [A], i > n}$ and $(s_a)_{a \in [A]\setminus \{b\}}$, Equation~\eqref{unimodal} implies that
        \begin{align*}
            & \left| \mathbb{E} \left[ \Xi_\pi \left( \left(\tilde{v}_a(t+s_a)\right)_{a\in [A]} \right) \big| \mathcal{G}_b \right] - \mathbb{E} \left[ \Xi_\pi \left( \tilde{v}_a(t+\Tilde{s}_a)_{a\in [A]} \right) \big| \mathcal{G}_b \right] \right| = \left| \mathbb{E} \left[ g(t + s_b - Y_b) - g(t - Y_b) \big| \mathcal{G}_b \right] \right| \\
            &
            = \left| \mathbb{E} \left[ g(t+s_b - (Y_b + s_b)) - g(t + s_b - Y_b) \big| \mathcal{G}_b \right] \right| \leq 2 Q\left( Y_b; s_b \right).
        \end{align*}
        Then, by Jensen's inequality, we get
        \begin{align*}
         & \left| \mathbb{E} \left[ \Xi_\pi \left( \left(\tilde{v}_a(t+s_a)\right)_{a\in [A]} \right) - \Xi_\pi \left( \tilde{v}_a(t+\Tilde{s}_a)_{a\in [A]} \right) \right] \right|
         \\
         &
         \leq
          \mathbb{E} \left[ \left| \mathbb{E} \left[ \Xi_\pi \left( \left(\tilde{v}_a(t+s_a)\right)_{a\in [A]} \right) - \Xi_\pi \left( \tilde{v}_a(t+\Tilde{s}_a)_{a\in [A]} \right) \big| \mathcal{G}_b \right] \right| \right] \leq 2 Q\left( Y_b; s_b \right) 
          \overset{\eqref{eq:factor 2}}{\leq}
           4 Q\left( \tau_{b}(n) ;s_b \right)
           .
        \end{align*}
        Applying this argument for all $b \in [A]$, by the triangle inequality,
        \begin{align*}
            \left| \mathbb{E} \left[ \Xi_\pi \left( \left(\tilde{v}_a(t+s_a)\right)_{a\in [A]} \right) - \Xi_\pi \left( \tilde{v}_a(t)_{a\in [A]} \right) \right] \right| \leq \sum_{b \in [A]}  2 Q\left( Y_b;s_b \right) \leq \sum_{b \in [A]}  4 Q\left( \tau_{b}(n);s_b \right),
        \end{align*}
        which finishes the proof.
    \end{proof}

    \subsection{Proof of Corollary~\ref{cor:conj}}
\begin{proof}[Proof of Corollary~\ref{cor:conj}]
Consider a collection of independent random variables $(X^{\ssup{a}}_{j})_{a \in [A], j \in \mathbb{N}}$ such that each $X^{\ssup{a}}_{j} \sim \exponentialrv(f(j-1))$. Then, if one sets $v_{a}^{\text{in}} := u_{a}(0)$ for all $a \in [A]$, with 
\[\tau_{n} := \inf\left\{t \geq 0\colon \sum_{a \in [A]} v_{a}(t) \geq n +\sum_{a \in [A]} v_{a}^{\text{in}} \right\}, 
\]
the processes $(v_{a}(\tau_{n}))_{n \in \mathbb{N}, a \in [A]}$ and $(u_{a}(n))_{n \in \mathbb{N}, a \in [A]}$ are equal in distribution -- this is Rubin's construction, a consequence of properties of the exponential distribution. On the other hand, the assumption~\eqref{eq:div-variance} guarantees that the sum $\sum_{j=1}^{\infty} X^{s}_{j}$ diverges almost surely (see, for example,~\cite[Section~2.3]{fixation-leadership-growth-processes}). Since $X_j^s$ is distributed like $X_j^{\ssup{1}}-X_j^{\ssup{2}}$ and $X_j^{\ssup{1}} \geq 0$, this implies that $\sum_{j=1}^{\infty} X^{\ssup{1}}_{j} = \infty$ almost surely, so that 
\[
\lim_{n \to \infty} \tau_{n} = \infty,
\]
whilst, since the waiting times $X^{\ssup{1}}_{i}$ are always finite almost surely, for each $n \in \mathbb{N}$ we have $\tau_{n} < \infty$ almost surely. 

By applying Theorem~\ref{thm:main}, almost surely, for any permutation $\pi$ there exists a sequence $\left(t^{\pi}_{i}\right)_{i \in \mathbb{N}}$ such that $\lim_{i \to \infty} t^{\pi}_{i} = \infty$ and  
\[
\forall i \in \mathbb{N} \quad 
v_{\pi(1)}(t^{\pi}_{i}) \geq v_{\pi(2)}(t^{\pi}_{i}) \geq \cdots \geq v_{\pi(A)}(t^{\pi}_{i}).
\]
On the other hand, since the composition of values $\left(v_{a}(t) \right)_{a \in [A]}$ only changes at the times $\tau_{n}$, this implies that there exists a sequence of integer $(n^{\pi}_{i})_{i \in \mathbb{N}}$ with 
\[(v_{a}(\tau_{n^{\pi}_{i}})_{i \in \mathbb{N}})_{a \in [A]} = (v_{a}(t^{\pi}_{i}))_{a \in [A]}, \]
hence, for all $i \in \mathbb{N}$, almost surely 
\[
v_{\pi(1)}(\tau_{n^{\pi}_{i}}) \geq v_{\pi(2)}(\tau_{n^{\pi}_{i}}) \geq \cdots \geq v_{\pi(A)}(\tau_{n^{\pi}_{i}}). 
\]
Therefore, if $(u_{a}(n))_{n \in \mathbb{N}, a \in [A]}$ is a balls-in-bins process coupled to agree with $(v_{a}(\tau_{n}))_{n \in \mathbb{N}, a \in [A]}$, almost surely, for all $i \in \mathbb{N}$ 
\[
u_{\pi(1)}(n^{\pi}_{i}) \geq u_{\pi(2)}(n^{\pi}_{i}) \geq \cdots \geq u_{\pi(A)}(n^{\pi}_{i}). 
\]
The result follows. 
\end{proof}

\subsection{Counterexample: proof of Proposition~\ref{prop:counter}}
\begin{proof}[Proof of Proposition~\ref{prop:counter}]
Suppose that $X_j^{\ssup{a}}$ is exponentially distributed with parameter $f(j) = 1/(j!)^8$. Note that $\sum_{j=1}^\infty X_j^{s}$ diverges if and only if $\sum_{j=0}^{\infty} 1/f(j)^2 = \infty$, which is clearly the case here (one can also observe that the summands of the series are unbounded since $|X_j^{s}|$ converges to infinity in probability). 
For $a \in [A], k \in \N$, define the event $A_k^{\ssup{a}}$ by
\begin{equation*}
    A_k^{\ssup{a}} \coloneqq \left\{ \sum_{j=1}^{k-1} X_j^{\ssup{a}} \leq k^3 ((k-1)!)^8 \right\} \cap \left\{ X_k^{\ssup{a}} > k^3 ((k-1)!)^8 \right\} .
\end{equation*}
By Markov's inequality, and the density of an exponential random variable, for all $k \geq 1$, the probability that the event $A_k^{\ssup{a}}$ does not hold is bounded by
\begin{align*}
    1-\mathbb{P}\left( A_k^{\ssup{a}} \right) 
    &
    \leq \mathbb{P} \left( \sum_{j=1}^{k-1} X_j^{\ssup{a}} > k^3 ((k-1)!)^8 \right) + \mathbb{P} \left( X_k^{\ssup{a}} \leq k^3 ((k-1)!)^8 \right) 
    \\
    &
    \leq \frac{\mathbb{E}\left[ \sum_{j=1}^{k-1} X_j^{\ssup{a}} \right]}{ k^3 ((k-1)!)^8 } + \int_0^{k^3 ((k-1)!)^8} \frac{1}{(k!)^8} e^{- \frac{1}{(k!)^8} t} \dd t
    \\
    &
    =
    \frac{\sum_{j=1}^{k-1} (j!)^8}{ k^3 ((k-1)!)^8 } + \int_0^{\frac{k^3 ((k-1)!)^8}{(k!)^8}}  e^{-t} \dd t
    \leq
    \frac{k ((k-1)!)^8}{ k^3 ((k-1)!)^8 } + \int_0^{\frac{1}{k^5}}  1\dd t \leq \frac{2}{k^2}.
\end{align*}
The Borel-Cantelli-lemma, and taking a finite intersection of almost sure events then implies that
\begin{align*}
    \mathbb{P} \left( \bigcap_{a \in [A]} \left\{  A_k^{\ssup{a}} \text{ for all but finitely many $k$}  \right\} \right) = 1.
\end{align*}
Thus, for any $\eps > 0$, there exists $K \in \mathbb{N}$ such that, with probability at least $1- \eps$, the events $A_k^{\ssup{a}}, a \in [A]$, occur for all $k \geq K$. Let $t > (K!)^8$, and let $m > K$ be so that
\begin{equation*}
    m^3 ((m-1)!)^8 < t \leq (m+1)^3 (m!)^8.
\end{equation*}
Since the events $A_m^{(a)}, A_{m+1}^{(a)}, a \in [A]$, hold, we get that
\begin{align*}
    & \sum_{j=1}^{m-1} X_j^{\ssup{a}} \leq m^3 ((m-1)!)^8 < t, \quad \text{ hence } \quad v_a(t) \geq m-1 \quad \text{ for all } a \in [A],
    \\
    & \text{whilst} \quad \sum_{j=1}^{m+1} X_j^{\ssup{a}} > X_{m+1}^{\ssup{a}} \geq (m+1)^3 (m!)^8 \geq t, \quad \text{hence} \quad v_a(t) \leq m \quad \text{ for all } a \in [A].
\end{align*}
In particular, with probability at least $1- \eps$, we get that $|v_a(t)-v_b(t)| \in \{0,1\}$ for all $a,b \in [A], t > (K!)^8$. Since $A \geq 3$ this implies \eqref{eq:strong-conjecture} is not satisfied. Since $\eps >0$ can be made arbitrarily small, the proposition follows. 
\end{proof}

\section*{Acknowledgements}
TI is funded by Deutsche Forschungsgemeinschaft (DFG) through DFG Project no. $443759178$. We thank an anonymous reviewer for careful reading and helpful comments which, in particular, led to Proposition~\ref{prop:counter}, answering a question of the reviewer.

\bibliographystyle{abbrv}
\bibliography{refs}

\end{document}